\documentclass[	10pt,	BCOR=12mm,	DIV=12,headinclude,headheight=12pt,oneside,onecolumn,openany,parskip=half,table,]{amsart}

\usepackage[utf8]{inputenc}
\usepackage[T1]{fontenc}
\usepackage{amsmath, amssymb, amsfonts, amsthm, mathtools,  mathrsfs}
\usepackage{graphicx}
\usepackage{tikz-cd}
\usepackage{lmodern}
\usepackage {dsfont }
\usepackage{caption}
\usepackage{subcaption}
\usepackage{todonotes}

\usepackage{multirow}
\usepackage{array}
\usepackage{tabularx}
\usepackage{mathpazo}
\usepackage{microtype}
	\usepackage[naustrian,english]{babel}
\usepackage[absolute]{textpos}
\usepackage{csquotes}
\usepackage{lastpage}
\usepackage{booktabs,colortbl,xcolor}
\usepackage{graphicx,wrapfig}
\usepackage[section]{placeins} 
\usepackage{float} 
\usepackage{enumitem}
\usepackage{subfiles}
\usepackage{scrhack}
\usepackage{hyperref}
\newtheorem{theorem}{Theorem}[section]
\newtheorem{lemm}[theorem]{Lemma}

\newtheorem{corr}[theorem]{Corollary}

\usepackage{enumitem} 

\allowdisplaybreaks
\makeatletter
\@namedef{subjclassname@2020}{%
	\textup{2020} Mathematics Subject Classification}
\makeatother

\title[Primitive elements of $\mathds{F}_{q^r}$ avoiding affine hyperplanes for $q=4$ and $q=5$.]{Primitive elements of finite fields $\mathds{F}_{q^r}$ avoiding affine hyperplanes for $q=4$ and $q=5$.}
\author[P. A. Grzywaczyk]{Philipp A. Grzywaczyk}
\address{Institut für Algebra\\Technische Universität Dresden\\
	01069 Dresden, Germany}
\urladdr{https://tu-dresden.de/mn/math/algebra}
\email{philipp$\_$alexander.grzywaczyk@tu-dresden.de}
\author[A. Winterhof]{Arne Winterhof}
\address{Johann Radon Institute for Computational and Applied Mathematics\\ Austrian Academy of Sciences\\
	4040 Linz, Austria}
\urladdr{https://www.oeaw.ac.at/ricam/}
\email{arne.winterhof@ricam.oeaw.ac.at}
\date{\today}
\thanks{The first author has been funded by the European Research Council (Project POCOCOP, ERC Synergy Grant 101071674). Views and opinions expressed are however those of the authors only and do not necessarily reflect those of the European Union or the European Research Council Executive Agency. Neither the European Union nor the granting authority can be held responsible for them.}

\usepackage[headsepline,footsepline]{scrlayer-scrpage}
\makeatletter
\let\thetitle\@title
\makeatother
\makeatletter
\let\theauthor\@author
\makeatother
\ohead{\footnotesize P. A. Grzywaczyk, A. Winterhof\\ \ }
\ihead{\footnotesize Primitive elements of $\mathds{F}_{q^r}$ avoiding affine \\hyperplanes for $q=4$ and $q=5$.}
\begin{document}
\bibliographystyle{amsalpha}
\maketitle
\begin{abstract}
	For a finite field $\mathds{F}_{q^r}$ with fixed $q$ and $r$ sufficiently large, we prove the existence of a primitive element outside of a set of $r$ many affine hyperplanes for $q=4$ and $q=5$.
	This complements earlier results by Fernandes and Reis for $q\ge 7$. 
	For $q=3$ the analogous result can be derived from a very recent bound on character sums of Iyer and Shparlinski.
	 {For $q=2$ the set consists only of a single element, and such a result is thus not possible.} 
\end{abstract}



\section{Introduction}
Let $q$ be a prime power and $r\geq 2$ be an integer. The finite field $\mathds{F}_{q^r}$ with~$q^r$ elements can be considered an $r$-dimensional $\mathds{F}_q$-vector space. By fixing an $\mathds{F}_q$-basis $\mathfrak{B}=\{\beta_1,...,\beta_r\}$ of $\mathds{F}_{q^r}$, we can write each element $\alpha\in\mathds{F}_{q^r}$ uniquely as \[\alpha=\sum_{i=1}^r a_i \beta_i, \quad a_i\in\mathds{F}_q.\] 

In \cite{f-reis}, Fernandes and Reis developed the idea of removing $r$ $\mathds{F}_q$-affine hyperplanes in general position from $\mathds{F}_{q^r}$ and, for $q\geq 7$, proved concrete and asymptotic results on the existence of primitive elements in the set $$\mathcal{S}_\mathcal{C}^\ast := \mathds{F}_{q^r}^{ {*}} \setminus\bigcup_{i=1}^r\mathcal{A}_i,$$ where $\mathcal{C}=\{\mathcal{A}_1,...,\mathcal{A}_r\}$ is a set of $\mathds{F}_q$-affine hyperplanes in general position, i.e.\ 
 { for} $1 \leq k \leq  {r}$, the intersection of any $k$ distinct elements of $\mathcal{C}$ is an $\mathds{F}_q$-affine space of dimension $r - k$. 

 {Since the results do not depend on the concrete value of the elements $c_1,...,c_r$, we consider them arbitrary but fixed for the remainder of this paper.

It is evident that  { the hyperplanes 
\begin{equation}\label{Aj}\mathcal{A}_j=\left\{\sum_{i=1}^r a_i\beta_i | a_j = c_j\right\},\quad c_j \in \mathds{F}_q,\quad j=1,2,\ldots,r,
\end{equation} are} 
in general position  {, since the set $\{\beta_1,\ldots,\beta_r\}$ constitutes a basis.

Conversely,} any set $\mathcal{A}_1,...,\mathcal{A}_r$ of affine hyperplanes in general position can be written in  { this way
for some basis $\{\beta_1,\ldots,\beta_r\}$.  More precisely, by definition, all $r$ hyperplanes in general position intersect in a single point $\gamma$ and this intersection is of dimension $0$. The~$r$ different intersections 
$$\bigcap_{\substack{i=1\\i\ne j}}^r {\mathcal{A}}_i,\quad j=1,2,\ldots,r,$$
of each $r-1$ hyperplanes are of dimension~$1$ and thus lines 
$${\mathcal{L}}_j=\{a\beta_j+\gamma: a \in \mathbb{F}_q\}$$ for some $\beta_j\in \mathbb{F}_{q^r}^*$, $j=1,...,r$. 
Obviously, for $i\ne j$ we have $\mathcal{L}_i\subseteq \mathcal{A}_j$, thus
$$\left\{\sum_{\substack{i=1\\ i\ne j}}^r a_i \beta_i+\gamma : a_i\in \mathbb{F}_q\right\}\subseteq \mathcal{A}_j,\quad j=1,\ldots r,$$
and it remains to show that the dimensions of both sides are the same, that is, $\beta_1,\ldots,\beta_r$ are linearly independent.
If $\beta_1,\ldots,\beta_r$ were linearly dependent, then for some $j$ and $a_i\in \mathbb{F}_q$ we would get
$$\beta_j=\sum_{\substack{i=1\\ i\ne j}}^r a_i\beta_i$$
and
$\mathcal{L}_j\subseteq \mathcal{A}_j.$
However, 
$$\mathcal{L}_j \cap \mathcal{A}_j=\bigcap_{i=1}^r \mathcal{A}_i=\gamma,$$
a contradiction.
Write $$\gamma=\sum_{j=1}^r c_j \beta_j, \quad c_j\in \mathbb{F}_q.$$
Now 
$$\mathcal{A}_j=\left\{\sum_{\substack{i=1\\ i\ne j}}^r a_i \beta_i+\gamma : a_i\in \mathbb{F}_q\right\}=\left\{\sum_{i=1}^r a_i\beta_i: a_j=c_j\right\},\quad j=1,2,\ldots,r.$$
Hence, for the rest of the paper we may assume \eqref{Aj}.}}

In particular, Fernandes and Reis proved for a fixed $q\ge 7$ the existence of a primitive element in $\mathcal{S}_\mathcal{C}^\ast$ if $r$ is sufficiently large.
We complement this result and obtain analogous existence theorems for primitive elements for $q=4$ and $q=5$. Our approach relies on using a different bound for the character sum 
$$s(\mathcal{S}_{\mathcal{C}}^\ast,\chi):=\left| \sum_{y\in\mathcal{S}_{\mathcal{C}}^\ast} \chi(y) \right|$$ 
appearing in the proof. 

A very recent bound of Iyer and Shparlinski \cite{is} can be used to fix the case $q=3$.  {The case $q=2$ is special, since removing $r$ affine hyperplanes from $\mathds{F}_{2^r}$ yields a set with only one element.}

In Section \ref{sec2} we prove a new upper bound for the character sum over the set in question, $\mathcal{S}_\mathcal{C}^\ast$. In Section \ref{sec3}, we obtain results for the existence of primitive elements in $\mathcal{S}_\mathcal{C}^\ast$.

\section{Preliminaries}\label{sec2}
A very important tool to study the existence of primitive elements in subsets of the
multiplicative group of finite fields is the
formula of Vinogradov, which was first published by Edmund Landau in \cite[p. 179,
Thm.\ 496]{landau}, see also \cite[Prop.\ 10.2.5]{HJ}.
\begin{lemm}\label{vinogradov}
 For $U\subseteq \mathds{F}_q^\ast$, the number $\mathcal{P}(U)$ of primitive
	elements in $U$ is given by
	\[ \mathcal{P}(U) = \frac{\varphi(q-1)}{q-1} \sum_{d|(q-1)} \left(\frac{\mu
		(d)}{\varphi(d)} \sum_{\substack{\chi\in\widehat{\mathds{F}}_q^\ast\\ ord\chi =
			d}} \sum_{x\in U} \chi(x)\right),\]
		where $\varphi$ denotes Euler's totient function, $\mu$ the Möbius-function and $\widehat{\mathds{F}}_q^\ast$ the group of multiplicative characters of $\mathds{F}_q$.
\end{lemm}
At the center of this formula we find a character sum over the respective subset. Notice that we can also write the set $\mathcal{S}_\mathcal{C}^\ast=\mathds{F}_{q^r} \setminus
\bigcup_{i=1}^r\mathcal{A}_i$ as 
$$\mathcal{S}_\mathcal{C}^\ast=\left\{\left.\sum_{i=1}^r a_i\beta_i\right|a_i\neq c_i\right\},$$ where each $c_i\in \mathbb{F}_q$ defines an affine hyperplane $\mathcal{A}_i$. We adapt the upper bound for character sums over sparse elements from \cite[Thm. 2.5]{msw} to this problem. Note that \cite[Thm.\ 2.5]{msw} deals only with $q=2$, but the proof can be adjusted to any $q$.
For the convenience of the reader, we prove this adapted version here for our special situation.
\begin{theorem}\label{char_sum_pg}
    Take
	$\mathcal{C}=\{\mathcal{A}_1,...,\mathcal{A}_r\}$ a set of $r$
	$\mathds{F}_q$-affine hyperplanes in $\mathds{F}_{q^r}$ defined by $c_1,\ldots,c_r$. Then for
	$\mathcal{S}_\mathcal{C}^\ast=\mathds{F}_{q^r} \setminus
	\bigcup_{i=1}^r\mathcal{A}_i$ and a non-trivial multiplicative character $\chi$ of
	$\mathds{F}_{q^r}$ we have 
    \[s(\mathcal{S}_\mathcal{C}^\ast,\chi) \leq \sqrt{3}
	(q-1)^{r/2} q^{\lceil3r/4\rceil/2}.\]
\end{theorem}
\begin{proof}
	Recall that we can write every $\alpha\in\mathcal{S}_\mathcal{C}^\ast$ as
	$\alpha=\sum_{i=1}^r a_i\beta_i$, with $a_i\neq c_i$. We can decompose $\mathcal{S}_\mathcal{C}^\ast$ into the sum of two sets
	\[\mathcal{V} =  \left\{\left.\sum_{i=1}^k a_i\beta_i\right|a_i\neq c_i\right\}  \] and
	\[\mathcal{W} = \left\{\left.\sum_{i=k+1}^r a_i\beta_i\right|a_i\neq c_i\right\}\] for a
	fixed value $k$ with $1\le k\le r$, which we will determine later. We get \begin{align*}
	s(\mathcal{S}_\mathcal{C}^\ast,\chi) &= \left| \sum_{v\in\mathcal{V}}
	\sum_{w\in\mathcal{W}} \chi(v+w)\right| \leq \sum_{v\in\mathcal{V}} 1 \left|
	\sum_{w\in\mathcal{W}} \chi(v+w)  \right|.\\
	\intertext{By applying the Cauchy-Schwarz inequality, we get:}
	s(\mathcal{S}_\mathcal{C}^\ast,\chi) &\leq |\mathcal{V}|^{1/2}
	\left(\sum_{v\in\mathcal{V}}\left|\sum_{w\in\mathcal{W}}
	\chi(v+w)\right|^2\right)^{1/2}.\\
	\intertext{We extend the sum over $\mathcal{V}$ to a sum over the linear subspace} 
		\mathcal{V}'&=\left\{\sum_{i=1}^k a_i\beta_i|a_i\in\mathds{F}_q\right\}\\
    \intertext{and get}
	s(\mathcal{S}_\mathcal{C}^\ast,\chi)&\leq (q-1)^{k/2} \left( \sum_{w_1,w_2\in\mathcal{W}}
	\left|\sum_{v\in\mathcal{V}'} \chi(v+w_1)\overline{\chi}(v+w_2)\right|
	\right)^{1/2}.
	\intertext{The absolute value of the sum over $\mathcal{V}'$ is equal to $q^k$ if $w_1=w_2$ and
		at most $2q^{r/2}$ if $w_1\neq w_2$, see \cite[p.8]{w2001}. Hence, we get}
	s(\mathcal{S}_\mathcal{C}^\ast,\chi)	&\leq (q-1)^{k/2} \left( 2(q-1)^{2(r-k)} q^{r/2} + (q-1)^{r-k}q^k
	\right)^{1/2}\\
	&\leq (q-1)^{r/2} \left( 2q^{3r/2 - k} + q^k \right)^{1/2}.\end{align*}
	By choosing $k=\lceil3r/4\rceil$, we get the desired bound.
\end{proof}
Note that for $q\leq 5$, Theorem \ref{char_sum_pg} is asymptotically sharper than the upper bound in \cite[Theorem 3.1]{f-reis},
that is,
\[ \sqrt{3} (q-1)^{r/2} q^{\lceil3r/4\rceil/2} < (2^r-1) q^{r/2}\] for $r$
	large enough. Note also that for $q\le 5$ the right hand side is larger than $(q-1)^r$ and the character sum bound in \cite{f-reis} is trivial.\\ 
 
We note that our character sum bound can be easily extended to hybrid character sums with polynomial arguments along the same lines as \cite{is,msw}. 

The following bound will be useful in later proofs:
\begin{lemm}\label{robin}
	For a positive integer $t$, let $W(t)$ denote the number of squarefree divisors of $t$. For $t\geq3$, we have \[W(t) < t^{0.96/\log\log t} .\]
\end{lemm}
\begin{proof}
	This is a direct consequence of \cite[Theorem 11]{robin}.
\end{proof}
\section{Results for Primitive Elements Avoiding Affine Hyperplanes}\label{sec3}
\begin{theorem}\label{prim_el_pg}
	Let $\mathcal{C} = \{\mathcal{A}_1, ..., \mathcal{A}_r\}$ be a set of
	$\mathds{F}_q$-affine hyperplanes of $\mathds{F}_{q^r}$ in general position, and
	$\mathcal{P}(\mathcal{S}_\mathcal{C}^\ast)$ the number of primitive elements in
	$\mathcal{S}_\mathcal{C}^\ast$. Then we obtain: \[\frac{q^r-1}{\varphi(q^r-1)}
	\mathcal{P}(\mathcal{S}_\mathcal{C}^\ast) >
	(q-1)^r-\sqrt{3}(q-1)^{r/2}q^{\lceil 3r/4\rceil /2}\cdot W(q^r-1) \ ,\] where $W(q^r-1)$ is the
	number of squarefree divisors of $q^r-1$.
\end{theorem}
\begin{proof}
	 {By Vinogradov's formula, Lemma~\ref{vinogradov}, we obtain \[\frac{q^r-1}{\varphi(q^r-1)}\mathcal{P}(\mathcal{S}_\mathcal{C}^\ast) = \sum_{d|(q^r-1)} \frac{\mu(d)}{\varphi(d)} \sum_{\substack{\chi\in\widehat{\mathds{F}}_{q^r}^\ast\\ \text{ord}(\chi) =
			d}} \sum_{\omega\in\mathcal{S}_\mathcal{C}^\ast} \chi(\omega) \ .\]
   Separating the contribution of the trivial character, resolving the sums and using the character sum estimate from Theorem \ref{char_sum_pg} we get
 \begin{align*}
     \frac{q^r-1}{\varphi(q^r-1)}\mathcal{P}(\mathcal{S}_\mathcal{C}^\ast) &\geq (q-1)^r - 1 - \sum_{\substack{d|(q^r-1)\\d \neq 1\\ \mu(d)\neq 0}} \frac{1}{\varphi(d)} \sum_{\substack{\chi\in\widehat{\mathds{F}}_{q^r}^\ast\\ \text{ord}(\chi) = d}} (\sqrt{3}
	(q-1)^{r/2} q^{\lceil3r/4\rceil/2})\\
    &= (q-1)^r - 1 - \sqrt{3}
	(q-1)^{r/2} q^{\lceil3r/4\rceil/2} \cdot \sum_{\substack{d|(q^r-1)\\d \neq 1\\ \mu(d)\neq 0}} 1\\
    &> (q-1)^r-\sqrt{3}(q-1)^{r/2}q^{\lceil 3r/4\rceil /2}\cdot W(q^r-1),
 \end{align*}
 which completes the proof.}
\end{proof}
By shifting this inequality and applying the bound from Lemma \ref{robin} for
$W(q^r-1)$, we get the following condition for the existence of a primitive
element in
$\mathcal{S}_\mathcal{C}^\ast$:\begin{equation}\label{prim_el_ineq_pg}
\left( \sqrt{3}\cdot q^{0.96r/\log\log q^r} \right)^{1/r} \leq
\frac{(q-1)^{1/2}}{q^{\lceil3r/4\rceil/2r}} .
\end{equation} In this form, we see that the left hand side of
the inequality is continously decreasing. The right hand side does converge for large values of $r$. We obtain the following
values for its limit:\\
\ \\
	{\large \begin{tabularx}{\textwidth} {
			| >{\centering\arraybackslash}X 
			|| >{\centering\arraybackslash}X
			| >{\centering\arraybackslash}X 
			| >{\centering\arraybackslash}X | }
		\hline 
		$q$ & $3$ & $4$ & $5$ \\
		\hline
		$\lim\limits_{r\rightarrow\infty}\frac{(q-1)^{1/2}}{q^{\lceil3r/4\rceil/2r}}$ & 
		$\frac{\sqrt{2}}{3^{3/8}}$ &
		$\frac{\sqrt{3}}{2^{3/4}}$ & $\frac{2}{5^{3/8}}$ \\
		\hline
		{\scriptsize Approx. Num. Value} & 
		$0.936687$
		& $1.02988$ 
		& $1.09375$ \\
		\hline
	\end{tabularx}}
\ \\ \ \\
For the left hand side of inequality \eqref{prim_el_ineq_pg}, we obtain the
following result:
\begin{lemm}\label{f_1_inf}
 \[\lim_{r\rightarrow\infty} \left(
	 \sqrt{3}\cdot q^{0.96r/\log\log q^r} \right)^{1/r} = 1.\]
\end{lemm}
\begin{proof}
	\begin{align*}
	\lim_{r\rightarrow\infty} \left( \sqrt{3}\cdot q^{0.96r/(\log\log (q^r))}
	\right)^{1/r} &= \lim_{r\rightarrow\infty} \ \underbrace{(\sqrt{3})^{1/r}}_{\rightarrow 1}
	\cdot \left( q^ {0.96r / (\log\log(q^r)) }  \right)^{1/r}\\
	&= \lim_{r\rightarrow\infty} \   q^ {0.96 / (\log\log(q^r)) } = 1.
	\end{align*}
	Since $\log\log(q^r)\rightarrow\infty$, the exponent approaches zero, and thus
	the term approaches $1$.
\end{proof}
Combining this Lemma with the values from the above table directly gives
us the following result:
\begin{theorem}
	Let $r \geq 2$ be a positive integer and let $\mathcal{C} = \{\mathcal{A}_1,
	..., \mathcal{A}_r\}$ be a set of $\mathds{F}_q$-affine hyperplanes of
	$\mathds{F}_{q^r}$ in general position. Then the set
	$\mathcal{S}_\mathcal{C}^\ast = \mathds{F}_{q^r} \setminus
	\bigcup_{i=1}^r\mathcal{A}_i$ contains a primitive element of $\mathds{F}_{q^r}$
	if $q=4$ or $q=5$ and $r$ is large enough.\qed
\end{theorem}

It is evident from the table that this bound does not yield any result for
$q=3$, since the right hand side of Inequality \ref{prim_el_ineq_pg} converges to a value of
approximately $0.936687$, while the left hand side still approaches $1$ for
large values of $r$. However, a recent result by Iyer and Shparlinski allows us to obtain results for $q=3$ as well:
\begin{lemm}[{\cite[Cor. 3.3]{is}}]
Let $q=3$ and $\mathcal{A}=\{0,2\}$, let $\chi$ be a non-principal multiplicative character, and let $f(X)\in\mathds{F}_{q^r}(X)$ a rational function of degree at most $d$ which is not an $n$-th power of some rational function in $\overline{\mathds{F}_p}(X)$. Then we have \[\sum_{\nu\in\mathcal{S}_r(\mathcal{A})} \chi(f(\nu)) \ll 2^{\gamma r} ,\] where \[\gamma = 1-\frac{\log{3}}{\log{2}}\cdot \kappa_5 \left( \frac{\log{2}}{\log{3}}\right) = 0.99128 . . . \] and \[\mathcal{S}_r(\mathcal{A}) = \left\{\left.\sum_{i=1}^r a_i \beta_i \right| a_1,...,a_r\in\mathcal{A}\right\}.\]
\end{lemm}
The choice
$$f(X)=X+\sum_{i=1}^r (c_i-1)\beta_i$$
implies
$$\mathcal{S}_r(\mathcal{A})=\mathcal{S}_\mathcal{C}^\ast.$$
Employing this estimate for the character sum and combining with Theorem \ref{prim_el_pg} yields the condition\[\left(q^{0.96r/\log\log{q^r}}\right)^{1/r} \leq \frac{q-1}{2^\gamma} ,\] which obviously holds for $q=3$ and $r$ large enough.

Combining the above results with Theorem 1.1 from \cite{f-reis} allows us to state the following general result:
\begin{corr}
    Let $q\geq 3$ be a prime power and $r \geq 2$ an integer, and let $\mathcal{C} = \{\mathcal{A}_1,
	..., \mathcal{A}_r\}$ be a set of $\mathds{F}_q$-affine hyperplanes of
	$\mathds{F}_{q^r}$ in general position. Then the set
	$\mathcal{S}_\mathcal{C}^\ast = \mathds{F}_{q^r} \setminus
	\bigcup_{i=1}^r\mathcal{A}_i$ contains a primitive element of $\mathds{F}_{q^r}$
	if $r$ is large enough.\qed
\end{corr}

 {As mentioned in the introduction, this problem is trivial for $q=2$. In this case, a different approach may be used to get small sets containing primitive elements: One might employ the notion of $s$-sparse elements and give an estimate on the possible values of $s$ such that the set of $s$-sparse elements contains a primitive element. This can be done using a recent bound on character sums of M\'erai, Shparlinski, and the second author \cite{msw}.}

\bibliography{literature}{}

\end{document}